\documentclass[10pt,a4paper]{article}
\usepackage{graphicx}
\usepackage{amsmath}
\usepackage{amsfonts}
\usepackage{amssymb}
\usepackage{amsthm}
\usepackage{color}
\begin{document}

\def\R{\mathbb{R}}
\def\N{\mathbb{N}}
\def\H{\mathcal{H}}
\def\d{\textrm{div}}
\def\v{\textbf{v}}
\def\I{\hat{I}}
\def\B{\hat{B}}
\def\x{\hat{x}}
\def\y{\hat{y}}
\def\p{\hat{\phi}}
\def\r{\hat{r}}
\def\w{\textbf{w}}
\def\u{\textbf{u}}
\def\K{\mathcal{K}}
\def\Reg{\textrm{Reg}}
\def\s{\textrm{Sing}}
\def\sgn{\textrm{sgn}}

\newtheorem{lem}{Lemma}[section]
\newtheorem{defin}{Definition}[section]
\newtheorem{conj}{Conjecture}[section]
\newtheorem{rem}{Remark}[section]
\newtheorem{cor}{Corollary}[section]
\newtheorem{thm}{Theorem}[section]
\newtheorem{prop}{Proposition}[section]
\newtheorem{definition}{Definition}[section]
\newtheorem{con}{Conjecture}[section]
\newtheorem{Main}{Main Result}

\title{Constrained optimal rearrangement problem leading to a new type obstacle problem}

\author{Hayk Mikayelyan \thanks{Mathematical Sciences, University of Nottingham Ningbo, 199 Taikang East Road, Ningbo 315100, PR China \hfill Hayk.Mikayelyan@nottingham.edu.cn  } }

\maketitle

\begin{abstract}
An optimal rearrangement problem
in a cylindrical domain $\Omega=D\times (0,1)$
is considered, under the constraint that the force function
does not depend on the $x_n$ variable of the cylindrical axis.
This leads to a new type of obstacle problem in the cylindrical domain
$$
\Delta u(x',x_n) = \chi_{\{v>0\}}(x') + \chi_{\{v=0\}}(x') [\partial_\nu u (x',0) + \partial_\nu u (x',1)]
$$
arising from minimization of the functional
$$
\int_\Omega \frac{1}{2}|\nabla u(x)|^2+\chi_{\{v>0\}}(x')u(x)\,dx,
$$
where $v(x')=\int_0^1 u(x', t) dt $, and $\partial_\nu u$ is the exterior normal derivative of $u$
at the boundary. Several existence and
regularity results are proven and it is shown
that the comparison principle does not hold for minimizers.
\end{abstract}

\thanks{obstacle problem, rearrangements, 35R35, 49J20}
\maketitle

\section{Introduction}

\subsection{Background}

One of the classical problems in rearrangement theory is the minimization of the
functional
\begin{equation}\label{Dir-functional}
\Phi(f)=\int_\Omega |\nabla u_f|^2 dx,
\end{equation}
where $u_f$ is the unique solution of the Dirichlet boundary value problem
\begin{equation}\label{main}
\begin{cases} -\Delta u_f(x) = f (x)  & \mbox{in } \Omega,
 \\
 u_f=0  & \mbox{on } \partial\Omega,  \end{cases}
\end{equation}
and $f$ belongs to the rearrangement class
$$
\mathcal{R}(f_0)=\{f\in L^2(\Omega) \,\, | \,\, \mathcal{L}^n(\{f>\alpha\})=\mathcal{L}^n(\{f_0>\alpha\})\,\,\text{for all}\,\,\alpha\in\mathbb{R} \}.
$$
Here $\Omega$ is a bounded domain in $\mathbb{R}^n$ with piece-wise smooth boundary $ \partial\Omega$, $\mathcal{L}^n$ denotes the Lebesgue measure, and $f_0\in L^2(\Omega)$ is the so-called generator
function of the rearrangement class. In this paper we will always assume that $f_0=\chi_{\Omega_0}$ for some sub-domain $\Omega_0\subset\Omega$.

This minimization problem is related to stationary heat equation 
$$
\underbrace{\partial_t u}_{=0}-\Delta u(x) = f (x)
$$
in the domain
$\Omega$, which is under the action of the external heat source modeled by the 
force function $f$. The
boundary condition $u(x)=0$ for $x\in\partial\Omega$ models the constant boundary temperature 
on the boundary of $\Omega$. Different force functions $f$ result different heat distributions $u_f$.
The minimizer $\hat{f}$ of the functional 
(\ref{Dir-functional}) is the force function from a certain rearrangement class $\mathcal{R}$, 
which is resulting the most uniformly distributed heat $u_{\hat{f}}$.

The problem and its variations, such as the $p-$harmonic case, has been studied by various authors (see \cite{B1}, \cite{B2}, \cite{BM}, \cite{EL}, \cite{Kbook}), and the results, for this particular setting, can be formulated in the following theorem.

 \begin{thm}
 There exists a unique solution $\hat{f}\in\mathcal{R}(\chi_{\Omega_0})$ of the minimization problem (\ref{Dir-functional}). For the function
 $\hat{u}=u_{\hat{f}}$ there exists a constant $\alpha>0$ such that
 \vspace{2mm}

\hspace{3cm}$\bullet$ $0< \hat{u}\leq\alpha$ in $\Omega$,
  \vspace{1mm}

\hspace{3cm}$\bullet$ $\hat{f}=\chi_{\{\hat{u}<\alpha\}}$,
   \vspace{1mm}

\hspace{3cm}$\bullet$ $\hat{u}= \alpha$ in $\{\hat{f}=0\}$.
\vspace{2mm}

\noindent Moreover, the function $U=\alpha-\hat{u}$ is the minimizer of the functional
$$
J(w)=\int_\Omega |\nabla w|^2+ 2\max(w,0)dx,
$$
among functions $w\in W^{1,2}(\Omega)$ with boundary values $\alpha$ on $\partial \Omega$, and
solves the obstacle problem equation
$$
\Delta U=\chi_{\{U>0\}}.
$$
\end{thm}

We refrain from presenting here details about the obstacle problem, which is one of the classical free boundary problems (see \cite{CaffRevisited}).


\subsection{The problem in the cylindrical domain}

In many applications heating is implemented by heating elements which are straight rods. Those are
usually placed parallel to each other in a cylindrical container, which is very natural, since in 3D it is
highly problematic and expensive to place point-wise acting heating elements all over the domain.
\vspace{3mm}

 Motivated by this we will consider a barrel-like domain
$$
\Omega=D\times (0,1)\subset \mathbb{R}^{n-1}_{x'}\times\mathbb{R}_{x_n}.
$$ and will restrict ourselves on force
functions $f(x)=f(x')$, which do not depend on the $x_n$ variable.

\begin{defin}
Let  $L^2_D(\Omega)$ be the subspace of $L^2(\Omega)$ which consists of functions constant
w.r.t. $x_n$ variable
$$
L^2_D(\Omega)=\{ g \in L^2(\Omega)\,\, | \,\, \exists h\in L^2(D)\,\,\text{such that}\,\, g(x',x_n)=h(x')
\,\,\text{a.e. in}\,\, \Omega  \}.
$$
Let now
$$
\mathcal{R}_D(f_0)=
\{f\in L^2_D(\Omega) \,\, | \,\, \mathcal{L}^{n}(\{f>\alpha\})=\mathcal{L}^n(\{f_0>\alpha\})\,\,\text{for all}\,\,\alpha\in\mathbb{R} \}
\subset
\mathcal{R}(f_0)
$$
be the subclass of the rearrangement class
consisting only of functions, which do not depend on $x_n$ variable.

\noindent Further let
$\bar{\mathcal{R}}_D(f_0)$ be the $w^*$-closure of $\mathcal{R}_D(f_0)$ in $L^2(\Omega)$ (see Lemmas \ref{lem:rearr1}
and \ref{lem:rearr11}).
\end{defin}

\begin{rem}
Without introducing new notations, in the sequel we will interpret functions $h\in L^2(D)$
to be also defined as functions in $L^2(\Omega)$ simply as $$h(x',x_n):=h(x').$$
\end{rem}

\vspace{5mm}

In this paper we will consider the minimization problem
$$
\min_{f\in \bar{\mathcal{R}}_D}\Phi(f)
$$
where $\bar{\mathcal{R}}_D=\bar{\mathcal{R}}_D(\chi_{D_0})$, $D_0\subset D$, is the sub-class 
of the force functions $f$, which do not depend on $x_n$-variable.

There is also a mathematical novelty in this setting. First, there exits no minimizer in the rearrangement class $\mathcal{R}_D$,
which on practice means that the optimal heating cannot be achieved by 
a 0/1 distribution of external heating source (heating elements), as it was the case in the problem without constraint. This means that the 
minimizer will belong
to the weak-$*$ closure $\bar{\mathcal{R}}_D$ of 
$\mathcal{R}_D$ (see Lemmas \ref{lem:rearr1}
and \ref{lem:rearr11}).

Second, the corresponding function $\hat{u}=u_{\hat{f}}$
will be a solution of a new-type obstacle problem, where the obstacle is not acting point-wise and the
Neumann derivative of the function is present on the right hand side of the equation (see equation (\ref{main2})).
In addition we prove that the solutions to equation (\ref{main2}) do not satisfy the comparison principle (see Section \ref{sec:compprinc}).

\vspace{2mm}

In Section \ref{sec:main} we will formulate the main results of the paper, in Section \ref{sec:prelim}
we will introduce some known results and prove technical lemmas. The results related to the optimal
rearrangement problem are presented in Section \ref{sec:rearr}, while the properties of the minimizer
to the new type obstacle problem can be found in Section \ref{sec:fbp}.
The proofs mainly combine two approaches. In Section \ref{sec:rearr} we adapt the methods 
developed by Burton and co-authors in our setting, while in Section \ref{sec:fbp} we 
use techniques known 
from the theory of non-linear partial differential equations to show the regularity of solutions.

\subsection{acknowledgement}
The author is grateful to Behrouz Emamizadeh for introducing the theory of rearrangements to him,
as well as, John Andersson and Henrik Shahgholian for insightful discussions.


\section{Main results}\label{sec:main}

From now on we will assume that the generator function
of the rearrangement class is a characteristic function $f_0(x')=\chi_{D_0}(x')$,
where $D_0\subset D$. The functions $u_f$ and $v_f$ are defined in (\ref{main}) and (\ref{vdef}). We will also
mostly skip writing $\chi_{D_0}$ in $\bar{\mathcal{R}}_D= \bar{\mathcal{R}}_D(\chi_{D_0})$
and ${\mathcal{R}}_D= {\mathcal{R}}_D(\chi_{D_0})$


\begin{thm}\label{thm:minim}
The relaxed minimization problem
$$
\min_{f\in \bar{\mathcal{R}}_D}\Phi(f)
$$
has a unique solution $\hat{f}\in \bar{\mathcal{R}}_D\setminus\mathcal{R}_D$,
$\hat{f}>0$ in $D$, and there exists a constant $\alpha>0$ such that
\begin{equation*}
\hat{v}(x'):=v_{\hat{f}}(x')=\int_0^1 u_{\hat{f}}(x',t)dt \leq\alpha,
\end{equation*}
\begin{equation*}
\{ \hat{f} < 1\}\subset\{\hat{v}=\alpha\}
\end{equation*}
$$
\{\hat{v}<\alpha\}\subset\{\hat{f}=1\}.
$$
Moreover, the function $\hat{U}(x)=\alpha -u_{\hat{f}}$ is the minimizer of the convex functional
\begin{equation}\label{func}
J(U)=\int_\Omega |\nabla U|^2dx +2\int_D V^+ dx'
\end{equation}
among functions $U\in W^{1,2}(\Omega)$ such that $U-\alpha\in W_0^{1,2}(\Omega)$,
where
\begin{equation*}
V(x')=\int_0^1 U(x',x_n)dx_n.
\end{equation*}
\end{thm}

\bigskip


\begin{thm}\label{thm:fbp}
Consider the minimization of the following convex functional
\begin{equation}\label{funcobst}
J(u)=\int_\Omega |\nabla u|^2dx +2\int_D v^+ dx'
\end{equation}
among functions with prescribed boundary values $u\in g + W^{1,2}_0(\Omega) $, in a domain $\Omega=D\times (0,1)$, where
\begin{equation}\label{vdef}
v(x')=\int_0^1 u(x',x_n)dx_n.
\end{equation}
We further assume that $g$ is constant on $D\times \{0\}$ and $D\times \{1\}$ and that
\begin{equation}\label{condg}
0\leq g(x',x_n)\leq (1-x_n)g(x',0)+x_ng(x', 1)
\end{equation}
for all $x'\in\partial D$.

Then the functional $J$ has a unique minimizer $u$, which satisfies the inequality
$$
v(x')=\int_0^1 u(x',x_n)dx_n\geq 0
$$
and the equation
\begin{equation}\label{main2}
\Delta u(x) = \chi_{\{v>0\}}(x') + \chi_{\{v=0\}}(x') [\partial_\nu u (x',0) + \partial_\nu u (x',1)]
\,\,\,\text{in}\,\,\, \Omega.
\end{equation}
\end{thm}


\bigskip

We also prove the existence of weak second derivatives (Corollary \ref{cor:w22})
and
show that the comparison principle holds for the functions $v(x')$, but fails to
hold for the functions $u(x)$ (Theorem \ref{compV} and Remark \ref{rem:compfail}).
The regularity  of the free boundary is briefly discussed in Section \ref{sec:fbpdisc}.


\section{Preliminaries}\label{sec:prelim}
In this section we would like to present several mainly
classical results.

\begin{lem}
For the solutions of (\ref{main}) the following is true
\begin{equation}\label{propLap}
\Phi(f)=\int_\Omega f u_fdx=\int_\Omega |\nabla u_f|^2 dx=\sup_{u\in W_0^{1,2}(\Omega)}\int_\Omega
2fu-|\nabla u|^2dx.
\end{equation}
  \end{lem}
  \begin{proof}
  Proof follows from partial integration and basic calculus of variations.
\end{proof}

\begin{lem}\label{lem:lipschitz}
Let
$$
-\Delta u=h(x)\,\,\,\text{in}\,\,\, \Omega
$$ and $|h(x)|\leq M$ is an integrable function in $\Omega$. Further assume
$\sup_{\Omega}|u|\leq N$. Then
$$
\|u\|_{C^{1,\alpha}(\Omega')}\leq C(n,d)(M+N)
$$
where $ \Omega'\Subset\Omega$ and $d=\text{dist}(\Omega',\Omega^c)$.
\end{lem}
\begin{proof}
See Theorems 8.32, 8.34 in \cite{GT}.
\end{proof}

\begin{lem}\label{lem:lipschitzbdr}
Let $\Omega$ be a domain with $C^{1,\alpha}$ boundary and the functions $u$ and $h$
be as in Lemma \ref{lem:lipschitz}. Further assume $u=0$ on $\partial \Omega$. Then
$$
\|u\|_{C^{1,\alpha}(\Omega)}\leq C(n,\partial \Omega)(M+N).
$$
\end{lem}
\begin{proof}
See Theorems 8.33, 8.34 in \cite{GT}.
\end{proof}

\begin{lem}\label{lem:lipschitzB}
Let $\Omega=D\times(0,1)$ and
\begin{equation}
\begin{cases} -\Delta u(x) = h (x)  & \mbox{in } \Omega,
 \\
 u=0  & \mbox{on } \partial\Omega,  \end{cases}
\end{equation}
and $|h(x)|\leq M$ is an integrable function in $\Omega$. Further assume
$\sup_{\Omega}|u|\leq N$. Then
$$
\|u\|_{C^{1,\alpha}(D'\times (0, 1))}\leq C(n,d)(M+N),
$$
where $d=\text{dist}(D',D^c)$.

Moreover, if $D$ has $C^{1,\alpha}$ boundary then
$$
\|u\|_{C^{1,\alpha}(\Omega)}\leq C(n,\partial D)(M+N).
$$
\end{lem}
\begin{proof}
Let us extend the function $u$ by the odd reflection into $\tilde{\Omega}=D\times (-1,1)$
\begin{equation}
\tilde{u}(x',x_n)=
\begin{cases} u(x', x_n)  & \mbox{if }\,\,\, x_n\geq 0,
 \\
- u(x', -x_n)  & \mbox{if }\,\,\, x_n<0.  \end{cases}
\end{equation}
Let us check that
$-\Delta \tilde{u}(x)=\tilde{h}(x)$ weakly in $D\times (-1,1)$
where
\begin{equation}
\tilde{h}(x',x_n)=
\begin{cases} h(x', x_n)  & \mbox{if }\,\,\, x_n> 0,
 \\
- h(x', -x_n)  & \mbox{if }\,\,\, x_n<0  \end{cases}
\end{equation}
is a bounded function.
\begin{equation}\label{eq:check1}
\int_{\tilde{\Omega}} \nabla \tilde{u}(x)\nabla \phi(x) d x=
\int_{\tilde{\Omega}} \nabla \tilde{u}(x)\nabla(\phi(x) \varphi_\delta(x_n)) d x+
\int_{\tilde{\Omega}} \nabla \tilde{u}(x)\nabla (\phi(x) (1-\varphi_\delta(x_n)))d x=I_1+I_2
\end{equation}
where
$$
\varphi_\delta(t)=
\begin{cases} 1  & \mbox{if }\,\,\, |t|<\delta/2,
 \\
 0  & \mbox{if }\,\,\, |t|>\delta  \end{cases}
$$
is an even function from $C^\infty_0(\mathbb{R})$ with values in $[0,1]$, such that $|\varphi'(t)|\leq 4/\delta$.
Let us now estimate the integrals on the right hand side of (\ref{eq:check1}).
\begin{multline}
I_1=\int_\Omega \nabla u\nabla [(\phi(x', x_n)-\phi(x', -x_n))\varphi_\delta(x_n) ] d x=\\
\int_\Omega h(x)[(\phi(x', x_n)-\phi(x', -x_n))\varphi_\delta(x_n) ]d x+\\\underbrace{
\int_{\partial \Omega} u(x)\partial_\nu[(\phi(x', x_n)-\phi(x', -x_n))\varphi_\delta(x_n) ]d \sigma}_{=0}
\to_{\delta\to 0}0,
\end{multline}
where we have used the continuity of $\phi\in C^\infty_0(\tilde{\Omega})$. On the other hand
\begin{equation}
I_2=\int_{\tilde{\Omega}}h(x)\phi(x', x_n)(1-\varphi_\delta(x_n)) d x\to_{\delta\to 0}
\int_{\tilde{\Omega}}h(x)\phi(x', x_n) d x.
\end{equation}

\bigskip

The proof follows now from Lemmas \ref{lem:lipschitz} and \ref{lem:lipschitzbdr}.
\end{proof}

\begin{lem}\label{lem:n-1dim}
Let
\begin{equation}
\begin{cases} -\Delta u(x) = f (x')  & \mbox{in } \Omega,
 \\
 u=0  & \mbox{on } \partial\Omega,  \end{cases}
\end{equation}
then
\begin{equation}\label{eq:sym12}
u(x',x_n)=u(x',1-x_n).
\end{equation}
and the function $v_f(x')=\int_0^1 u_f(x',x_n)dx_n$ satisfies the following equation

\begin{equation}\label{x'laplace}
\begin{cases}
-\Delta_{x'} v=f(x')+2\partial_\nu u(x',0)  & \mbox{in } D,
 \\
 v=0  & \mbox{on } \partial D.  \end{cases}
\end{equation}
\end{lem}
\begin{proof}
(\ref{eq:sym12}) follows from the uniqueness of the solution.

Let us take $\phi_\delta(x)=\psi(x')\varphi_\delta(x_n)$, where $\psi\in C_0^\infty(D)$
and
$$
\varphi_\delta(x_n)=\begin{cases}
\frac{1}{\delta(1-\delta)}x_n &\text{if} \,\,\,x_n\in(0,\delta)\\
\frac{1}{1-\delta} &\text{if} \,\,\,x_n\in(\delta,1-\delta)\\
\frac{1}{\delta(1-\delta)}-\frac{1}{\delta(1-\delta)}x_n &\text{if} \,\,\,x_n\in(1-\delta,1).
\end{cases}
$$

\begin{multline}
\int_D f(x')\psi(x')dx'=\int_\Omega f(x')\phi_\delta(x)dx=\int_\Omega \nabla u\nabla\phi_\delta dx\\
=\int_\Omega \varphi_\delta(x_n)\nabla'u(x)\nabla' \psi(x')dx+
\int_\Omega \psi(x')\partial_nu(x)\partial_n \varphi_\delta(x_n)dx.
\end{multline}
Passing to the limit as $\delta\to 0$ we obtain
$$
\int_\Omega \varphi_\delta(x_n)\nabla'u(x)\nabla' \psi(x')dx\to_{\delta\to 0}
\int_\Omega \nabla'u(x)\nabla' \psi(x')dx=\int_D \nabla'v(x)\nabla' \psi(x')dx'
$$
and using Lemma \ref{lem:lipschitzB}
\begin{multline}
\int_\Omega \psi(x')\partial_nu(x)\partial_n \varphi_\delta(x_n)dx=\\
\frac{1}{\delta(1-\delta)}\left[\int_D\int_0^\delta  \psi(x')\partial_nu(x) dx'dx_n -
 \int_D\int_{1-\delta}^1 \psi(x')\partial_nu(x) dx'dx_n\right]\to_{\delta\to 0}\\
\int_D \psi(x')[\partial_n u(x',0)-\partial_n u(x',1)] dx'.
\end{multline}
Thus
$$
\int_D \nabla'v(x)\nabla' \psi(x')dx'=
\int_D f(x')\psi(x')dx'-\int_D \psi(x')[\partial_n u(x',0)-\partial_n u(x',1)] dx'.
$$
From (\ref{eq:sym12}) we obtain $\partial_n u(x',0)=-\partial_n u(x',1)$.
\end{proof}

\begin{lem}\label{lem:rearr1}
Let $D_0\subset D$ and $\bar{\mathcal{R}}(\chi_{D_0})$ be the $w^*$-closure of $\mathcal{R}(\chi_{D_0})$ in $L^2(D)$. Then
$$
\bar{\mathcal{R}}(\chi_{D_0})=\{h\,\, | \,\, 0\leq h\leq 1,\,\,\text{and}\,\,\int_D hdx'=|D_0| \}
$$
is convex and weakly compact in $L^2(D)$. Moreover, the set of its extreme points
is
$$
ext(\bar{\mathcal{R}}(\chi_{D_0}))=\mathcal{R}(\chi_{D_0}).
$$
\end{lem}

\begin{proof}
See \cite{B1}, \cite{B2}, \cite{BR}, \cite{EL}.
\end{proof}

\begin{lem}\label{lem:rearr11}
Let $D_0\subset D$ and $\bar{\mathcal{R}}_D(\chi_{D_0})$ be the $w^*$-closure of $\mathcal{R}_D(\chi_{D_0})$ in $L^2(\Omega)$. Then
$$
\bar{\mathcal{R}}(\chi_{D_0})=\{h\in L^2_D(\Omega)\,\, | \,\, 0\leq h\leq 1,\,\,\text{and}\,\,\int_\Omega hdx'=|D_0| \}
$$
is convex and weakly compact in $L^2(\Omega)$. Moreover, the set of its extreme points
is
$$
ext(\bar{\mathcal{R}}_D(\chi_{D_0}))=\mathcal{R}_D(\chi_{D_0}).
$$
\end{lem}
\begin{proof}
Follows from Lemma \ref{lem:rearr1}.
\end{proof}

\begin{lem}\label{lem:rearr2}
The functional $\Phi$ (see (\ref{propLap})) is

(i) weakly sequentially continuous in $L^2$,

(ii) strictly convex,

(iii) G\^ateaux differentiable. Moreover, $\Phi'(f)$ can be identified with $2u_f$ if

we consider $\Phi$ in $L^2(\Omega)$ or $2v_f$ if we consider $\Phi$ in $L^2(D)$.
\end{lem}

\begin{proof}
The proof can be found in \cite{BM}.
\end{proof}

\begin{lem}\label{lem:rearr3}
For $f,g\in L^2_+(D)$ there exists
$\widetilde{f}\in \mathcal{R}(f)$ such that
functional
$$
\int_D \widetilde{f}g dx\leq \int_D hg dx,
$$
for all $h\in \bar{\mathcal{R}}(f)$.
\end{lem}

\begin{proof}
The proof can be found in \cite{B1}.
\end{proof}


\section{The constrained rearrangement problem}\label{sec:rearr}

\begin{proof}[Proof of Theorem \ref{thm:minim}]

By Lemmas \ref{lem:rearr11} and \ref{lem:rearr2}
$$
\min_{f\in \bar{\mathcal{R}}_D}\Phi(f)
$$
has a solution since $\bar{\mathcal{R}}_D$ is weakly compact and $\Phi$ is weakly continuous. Further,
the minimizer
$\hat{f}\in \bar{\mathcal{R}}_D$ is unique,
since $\Phi$ is strictly convex.

Let us now prove that $\hat{f}\notin\mathcal{R}_D $. The condition for the minimizer is
$$
0\in \partial \Phi(\hat{f})+ \partial\xi_{\bar{\mathcal{R}}_D}(\hat{f}),
$$
where $\partial \Phi$ is the sub-differential and
$$\xi_{\bar{\mathcal{R}}_D}(g)=
\begin{cases}
0  & \mbox{if } g\in \bar{\mathcal{R}}_D,
 \\
\infty  & \mbox{if } g\notin \bar{\mathcal{R}}_D,  \end{cases}
$$
see \cite{CLSWbook}.
This means that $-2\hat{v}\in \partial\xi_{\bar{\mathcal{R}}_D}(\hat{f})$. Since
$$
\partial\xi_{\bar{\mathcal{R}}_D}(\hat{f})=\left\{ w\in L^2(D)\,\, : \,\,
\xi_{\bar{\mathcal{R}}_D}(f)-\xi_{\bar{\mathcal{R}}_D}(\hat{f})\geq
\int_D (f-\hat{f})wdx'
 \right\}
$$
we obtain
\begin{equation}\label{linear}
\int_D f \hat{v} dx'\geq \int_D \hat{f}\hat{v} dx'.
\end{equation}
for all $f\in\bar{\mathcal{R}}_D $. By Lemma \ref{lem:rearr3} there exists
$$
\widetilde{f}=\chi_{\widetilde{D}}\in ext (\bar{\mathcal{R}}_D )=\mathcal{R}_D,
$$
where $\widetilde{D}\subset D$, such that 
\begin{equation}\label{linear1}
\int_D \widetilde{f} \hat{v} dx' = \int_D \hat{f}\hat{v} dx'.
\end{equation}


{\bf Claim 1:}
\begin{equation}\label{eq:poqralpha1}
\alpha=\sup_{\widetilde{D}} \hat{v}\leq \inf_{D\setminus \widetilde{D}} \hat{v}.
\end{equation}
This follows from (\ref{linear}) and (\ref{linear1}). The idea of the proof is based on the bathtub principle (see \cite{LLbook}):
if (\ref{eq:poqralpha1})  fails to hold, we can rearrange the function $\widetilde{f}$
such that the integral $\int_D \widetilde{f}\hat{v}dx'$ decreases, by assigning the value $1$ to $\widetilde{f}$
where $\hat{v}$ is small and assigning the value $0$ to $\widetilde{f}$ where $\hat{v}$ is large
(for details see also \cite{EL}, equation (3.17)).


\bigskip
{\bf Claim 2:}
\begin{equation}\label{eq:poqralpha2}
\hat{f}=\widetilde{f}=\chi_{\widetilde{D}}=1, \,\,\text{in}\,\, \{\hat{v}<\alpha\}.
\end{equation}
The idea of the proof is the same as above: if (\ref{eq:poqralpha2}) fails to hold then
$$
\int_{D\setminus \widetilde{D}}\hat{f}dx'=\int_{ \widetilde{D}}(1-\hat{f})dx'>0,
$$
thus we can replace the function $\hat{f}$ by a function  $f\in\bar{\mathcal{R}}_D $ which has
larger values in $\{\hat{v}<\alpha\}\subset \widetilde{D}$ and smaller values in $D\setminus \widetilde{D}$.
As a result
\begin{equation*}
\int_D f \hat{v} dx'<\int_D \hat{f}\hat{v} dx',
\end{equation*}
which contradicts (\ref{linear}).

\bigskip


{\bf Claim 3:}
$$
\{\hat{v}>\alpha\}\subset D^\#:=\{\hat{f}=0\}.
$$
We know that $\int_D \widetilde{f}\hat{v}dx'=\int_D\hat{f}\hat{v}dx'$, and
\begin{equation}
\int_D \widetilde{f}\hat{v}dx'=
\int_{\{\hat{v}\geq \alpha\}} \widetilde{f}\hat{v}dx' +
\int_{\{\hat{v}<\alpha\}} \widetilde{f}\hat{v}dx'=
\int_{\{\hat{v}\geq\alpha\}} \hat{f}\hat{v}dx'
+\int_{\{\hat{v}<\alpha\}} \hat{f}\hat{v}dx'
=\int_D\hat{f}\hat{v}dx'.
\end{equation}
On the other hand $\int_{\{\hat{v}<\alpha\}} \widetilde{f}\hat{v}dx'=
\int_{\{\hat{v}<\alpha\}} \hat{f}\hat{v}dx'=
\int_{\{\hat{v}<\alpha\}} \hat{v}dx'$
and $\widetilde{f}=0$ on $\{\hat{v}>\alpha\}$.
This means that
\begin{equation}\alpha \int_{\{\hat{v}\geq \alpha\}} \hat{f}dx'=
\alpha \int_{\{\hat{v}\geq \alpha\}} \widetilde{f}dx'=
\int_{\{\hat{v}\geq \alpha\}} \widetilde{f}\hat{v}dx' =
\int_{\{\hat{v}\geq\alpha\}} \hat{f}\hat{v}dx'
\geq\alpha \int_{\{\hat{v}\geq \alpha\}} \hat{f}dx',
\end{equation}
where the last inequality will be strict if $\{\hat{v}>\alpha\}\cap \{\hat{f}>0\}$
has a positive measure.

\bigskip


{\bf Claim 4:}
$$
D^\# \,\,\,\text{has no interior.
Thus $\hat{v}\leq\alpha$.}
$$
From (\ref{x'laplace}) and the Hopf's lemma it follows that
$$
\Delta_{x'}\hat{v}(x')=-2\partial_\nu u(x',0)>0 \,\,\,\text{in}\,\,\,\text{int}(D^\#)
$$
and $\hat{v}\geq \alpha$ in $\text{int}(D^\#)$. This means that there exists $y\in \partial(\text{int}(D^\#))$
such that $\hat{v}(y)=\beta> \alpha$, which contradicts Claim 3 and continuity of $\hat{v}$.

\bigskip
{\bf Claim 5:}
$$
\hat{f}>0.
$$
We need to verify this only in $\text{int}(\{\hat{v}=\alpha\})$ where
$$
0=\Delta_{x'} \hat{v}=-
\hat{f}(x')-2\partial_\nu \hat{u}(x',0)
$$
and the outer normal derivative of $\hat{u}$ is not vanishing in $D$ by Hopf lemma.

\bigskip


{\bf Claim 6:}
$$
\hat{f}\notin {\mathcal{R}}_D= {\mathcal{R}}_D(\chi_{D_0}).
$$
This follows from the positivity of $\hat{f}$, since otherwise $\{\hat{f}=0\}\not=\emptyset$.

\bigskip


{\bf Claim 7:}
$\hat{U}=\alpha-\hat{u}$ minimizes the functional (\ref{func}).

\noindent From (\ref{propLap}) we can obtain that $\hat{U}$ minimizes the functional
$$
I(U)=\int_\Omega |\nabla U|^2 +2\hat{f}Udx=\int_\Omega |\nabla U|^2dx +2\int_D \hat{f}Vdx'
$$
among $U\in W^{1,2}(\Omega)$ such that $U=\alpha$ on $\partial \Omega$. For any such function $U$ we have
$$
J(U)\geq I(U)\geq I(\hat{U})=J(\hat{U}).
$$
\end{proof}

\section{New type of obstacle problem}\label{sec:fbp}


In this section we discuss the new type of obstacle problem introduced in Theorem \ref{thm:fbp}, where the obstacle is acting
not on the function $u$, but on the integral of $u$ with respect to $x_n$ variable. As a result,
the free boundary is not a level set for the function $u$.
\subsection{Existence of solutions}

\begin{proof}[Proof of Theorem \ref{thm:fbp}]
Observe that
$$
J(u)=\int_\Omega |\nabla u|^2dx +2\int_D v^+ dx'=\int_\Omega |\nabla u|^2 +2u\chi_{\{v>0\}} dx
$$
and take the variations $u_\epsilon(x)=u(x)+\epsilon \phi(x)$, where $\phi(x)\geq 0$.

For $\epsilon>0$ the variation gives
$$
2\int_\Omega \nabla u\nabla\phi dx +2\int_\Omega \chi_{\{ v\geq 0\}}\phi dx\geq 0
$$
and for $\epsilon<0$
$$
2\int_\Omega \nabla u\nabla\phi dx +2\int_\Omega \chi_{\{ v> 0\}}\phi dx\leq 0.
$$
Thus
$$
\int_\Omega \chi_{\{ v> 0\}}\phi dx \leq -\int_\Omega \nabla u\nabla\phi dx\leq \int_\Omega \chi_{\{ v\geq 0\}}\phi dx
$$
and the distribution $-\int_\Omega \nabla u\nabla\phi dx$ is a positive measure given by a
function identified with $\Delta u (x)$, such that
$$
 -\int_\Omega \nabla u\nabla\phi dx=\int_\Omega \Delta u(x)\phi(x)dx
$$
and
\begin{equation}\label{eq:bddLapl}
\chi_{\{v>0\}}\leq \Delta u \leq \chi_{\{v\geq 0\}}.
\end{equation}
\bigskip

{\bf Claim 1:} $\Delta u$ does not depend on $x_n$.

\noindent Let us consider the variation of the functional $J$ with test function
$u_\epsilon(x)=u(x)+\epsilon \phi(x)$
where
$\phi(x)=\varphi(x',x_n)-\varphi(x',x_n -a)$ such that
$\varphi(x',x_n), \varphi(x',x_n -a)\in C^\infty_0(\Omega)$.
Then $\int_0^1 \phi(x',x_n)dx_n=0$ and thus the second term of the functional does not contribute to the variation.
The contribution of the first term is
$$
\int_\Omega \nabla u\nabla \varphi(x',x_n) dx-\int_\Omega \nabla u\nabla \varphi(x',x_n-a) dx=0,
$$
which proves that $\Delta u$ does not depend on $x_n$.
\bigskip

{\bf Claim 2:}
\begin{equation}\label{x'laplace1}
\Delta_{x'} v=\Delta u(x')- [\partial_\nu u (x',0) + \partial_\nu u (x',1)]  \,\,\, \text{in } D.
\end{equation}
Follows from Lemma \ref{lem:n-1dim}.

\bigskip

{\bf Claim 3:} $$\{v<0\}=\emptyset.$$

Assume $\{v<0\}=D^*\subset D$. By continuity $D^*$ is open, $v=0$ on $\partial D^*$ and
$\Delta u=0$ in $D^*\times(0,1)$. By (\ref{x'laplace1})
$$
\Delta_{x'}v=- [\partial_\nu u (x',0) + \partial_\nu u (x',1)]  \,\,\, \text{in } D^*.
$$
Using the fact the boundary data $g$ is constant on $D\times \{0\}$ and $D\times \{1\}$, 
the condition (\ref{condg}), as well as the sub-harmonicity of $u$ we obtain 
by comparison principle
that  
$$
u(x',x_n)\leq (1-x_n)g(x',0)+x_ng(x',1), \,\,\,\text{for}\,\,\,x\in \Omega.
$$
Thus
$$
\Delta_{x'}v=- [\partial_\nu u (x',0) + \partial_\nu u (x',1)] \leq 0 \,\,\, \text{in } D^*,
$$
a contradiction. 
\bigskip

The equation (\ref{main2}) follows from (\ref{eq:bddLapl}) and (\ref{x'laplace1}).

\end{proof}


\subsection{Existence of weak second derivatives}




In this section we apply a difference quotient argument to show the existence of weak second derivatives. As in the case of the classical obstacle problem the method deals with the 
regularity of the function and not the regularity of the free boundary set.

\begin{lem}\label{thm:2nd_der1}
Let $ u $ be the minimizer of (\ref{funcobst}) in $\Omega=D\times (0,1)$ and $u$ is constant on $D\times \{0\}$
and on $D\times\{1\}$.
Then for any compact $\mathcal{C}\subset D$ there exists a constant $C$ depending only on $\text{dist}(\mathcal{C}, D^c)$
such that
\begin{equation}\label{ineqlemma}
\int_{\mathcal{C}\times (0,1)} \left| \frac{\nabla (u(x+eh)-u(x))}{h}  \right|^2dx\leq C
\int_\Omega \left| \frac{ u(x+eh)-u(x)}{h}  \right|^2dx
\end{equation}
for all $|h|<\text{dist}(\mathcal{C}, D^c)/2$ and all directions $e \bot e_n$.
\end{lem}
\begin{proof}
Let us take
$$
\phi(x)=\psi(x')^2(u(x+eh)-u(x)),
$$
where $\psi\in C_0^\infty (D)$, $0\leq \psi\leq 1$, $\psi(x')=1$ for $x'\in \mathcal{C}$, $\psi(x')=0$ for $\text{dist}(x',\mathcal{C})>\text{dist}(\mathcal{C},D^c)/2$
and $\nabla\psi\leq\frac{4}{\text{dist}(\mathcal{C},D^c)}$.
Observe that the boundary values of the function
$$
u(x)+t\phi(x)=t\psi(x')^2u(x+eh)+(1-t)\psi(x')^2u(x)
$$
are the same as of $u$.
Moreover, for $t\in(0,1)$
$$
\int_0^1 u(x)+t\phi(x)dx_n= t\psi(x')^2v(x+eh)+(1-t)\psi(x')^2v(x)\geq 0
$$
and we can consider the variations of the functional
\begin{equation}
I(u)=\int_\Omega |\nabla u|^2+2u \, dx.
\end{equation}
instead of (\ref{funcobst}). From
$$
J(u+t\phi)-J(u)=I(u+t\phi)-I(u)\geq 0
$$
we obtain
$$
0\leq \int_\Omega \nabla u\nabla \phi +\phi \, dx
$$
or
\begin{equation}\label{ineqint1}
\int_\Omega \nabla u(x)\nabla\left( \psi(x')^2(u(x+eh)-u(x)) \right)+ (\psi(x')^2(u(x+eh)-u(x)))dx\geq 0.
\end{equation}
Repeating the same argument as above for the function $u(x+eh)$ in a slightly shifted domain and using
the function $u(x)$ for constructing a perturbation we can obtain the inequality
\begin{equation}\label{ineqint2}
\int_\Omega \nabla u(x+eh)\nabla\left( \psi(x')^2(u(x)-u(x+eh)) \right)+ (\psi(x')^2(u(x)-u(x+eh)))dx\geq 0.
\end{equation}
adding (\ref{ineqint1}) and (\ref{ineqint2})
\begin{multline}
0\geq \int_\Omega \nabla \left( u(x+eh)-u(x) \right)\nabla\left( \psi(x')^2(u(x+eh)-u(x)) \right)dx\\
=\int_\Omega \psi(x')^2|\nabla \left( u(x+eh)-u(x) \right)|^2 dx+ \\
\int_\Omega\left(u(x+eh)-u(x) \right) 2\psi(x')\nabla\psi \nabla \left( u(x+eh)-u(x) \right)dx
\end{multline}
we arrive at
\begin{multline}\label{mult1st}
\int_\Omega \psi(x')^2|\nabla \left( u(x+eh)-u(x) \right)|^2 dx\leq 
-\int_\Omega 2\left[\left(u(x+eh)-u(x) \right) \nabla\psi\right]\cdot\left[ \psi(x')\nabla \left( u(x+eh)-u(x) \right)\right]dx.
\end{multline}
Now we use the inequality $2|{\bf x}\cdot{\bf y}|\leq 2|{\bf x}|^2+\frac{1}{2}|{\bf y}|^2$
to derive
\begin{multline}
-2\left[\left(u(x+eh)-u(x) \right) \nabla\psi\right]\cdot\left[ \psi(x')\nabla \left( u(x+eh)-u(x) \right)\right]\leq \\
2|\nabla\psi|^2|u(x+eh)-u(x)|^2+\frac{1}{2}\psi(x')^2|\nabla \left( u(x+eh)-u(x) \right)|^2
\end{multline}
and obtain from (\ref{mult1st})
\begin{equation}
\int_\Omega \psi(x')^2|\nabla \left( u(x+eh)-u(x) \right)|^2 dx\leq \\
4\int_\Omega |\nabla\psi|^2|u(x+eh)-u(x)|^2dx
\end{equation}
Taking $C=\frac{64}{(\text{dist}(\mathcal{C},D^c))^2}$ and dividing by $h^2$ we obtain (\ref{ineqlemma}).
\end{proof}

\begin{lem}\label{thm:2nd_der2}
Let $\Omega'\Subset \Omega$, $\Omega'_\delta=\{x\,:\,\text{dist}(x,\Omega')<\delta\}\subset\Omega$, $w\in L^2(\Omega)$ and
$$
\int_{\Omega_\delta}\left|\frac{w(x+e_jh)-w(x)}{h}  \right|^2dx\leq C
$$
for some constant $C$ and all $|h|<\delta$.

Then the weak derivative $\frac{\partial w}{\partial x_j}$ exists in $\Omega'$
and
$$
\int_{\Omega_\delta}\left|\frac{\partial w}{\partial x_j}  \right|^2 dx\leq C.
$$
\end{lem}

\begin{proof}
See Lemma 7.24 in \cite{GT}.
\end{proof}

\begin{lem}\label{thm:2nd_der3}
Assume $\Omega'\Subset \Omega$ and $u\in W^{1,2}(\Omega)$. Then there exists
 a constant $C>0$ depending on dimension only such that
$$
\int_{\Omega'} \left| \frac{ u(x+e_jh)-u(x)}{h}  \right|^2dx\leq C
\int_{\Omega}\left|\frac{\partial u}{\partial x_j}  \right|^2 dx
$$
for all $|h|<\text{dist}(\Omega',\Omega^c)$.
\end{lem}

\begin{proof}
See Lemma 7.23 in \cite{GT}.
\end{proof}

\begin{cor}\label{cor:w22}
$$
u\in W^{2,2}(D'\times(0,1)),\,\,\,\text{for any } D'\Subset D.
$$
\end{cor}
\begin{proof}
The existence $\frac{\partial^2 u}{\partial x_i\partial x_j}$ in $L^2(D'\times(\delta,1-\delta))$,
where $1\leq i \leq n-1$ and $1\leq j\leq n$, and integral bounds
follow from Lemmas \ref{thm:2nd_der1}-\ref{thm:2nd_der3}.
The existence and integral bounds for $\frac{\partial^2 u}{\partial x_n^2}$
follow from (\ref{eq:bddLapl}).

Now let us observe that because of constant boundary data on $D\times \{0\}$
and $D\times \{1\}$ we can extend the function to $D\times (-1, 2)$ similarly as
we have done it in Lemma \ref{lem:lipschitzB}. This is why we can let $\delta=0$.
\end{proof}



\subsection{The comparison principle}\label{sec:compprinc}

One of the interesting features of the functional $J$ is that the comparison principle fails to hold for the
minimizers $u$.

\begin{rem}\label{rem:compfail}
The comparison principle does {\bf not} hold for the functions $u_1$ and $u_2$ in Theorem \ref{compV}. Particularly, in the set
$\{v_2=0\}\subset \{v_1 = 0 \}$, where
\begin{equation}\label{compinteq}
\int_0^1 u_1(x',x_n)dx_n=
\int_0^1 u_2(x',x_n)dx_n=0
\end{equation}
but $u_1\equiv  \!\!\!\!\! \backslash \,\,\, u_2$.
\end{rem}
\begin{proof}
Assume the comparison principle does hold and $u_1\leq u_2$. Then from (\ref{compinteq}) it follows that $u_1\equiv u_2$
in $\{ v_2=0 \}\times (0,1)$.
Let us now consider the function $w= u_2 - u_1\geq 0$. By (\ref{main2})
\begin{equation}
\Delta w = \begin{cases} 0 & \mbox{in } (\{v_1>0\}\cup\{v_2=0\})\times(0,1),
 \\
 1-2\partial_\nu u_1  & \mbox{in } (\{v_1=0\}\setminus\{v_2>0\})\times(0,1),  \end{cases}
\end{equation}
and by (\ref{eq:bddLapl}) $\Delta w\geq 0$. Since the function $w$ is positive at the boundary and vanishes in the set where $u_1=u_2$ is is not constant and thus, by Hopf lemma, $\partial_n w > 0$ in $\{v_2=0\}\times\{0\}$. This contradicts the fact of $u_1\equiv u_2$ on $\{v_2=0\}\times (0,1)$.
\end{proof}

We would like to state the following open problem.

Here we prove that it holds for the functions $v$ with constant
boundary data.

\begin{conj}\label{compV}
Let $u_1$ and $u_2$ minimize (\ref{funcobst}) among functions with constant boundary data $\alpha_1$ and $\alpha_2$
respectively, and $0<\alpha_1<\alpha _2$. Then
$$
v_1(x')\leq v_2(x')
$$
for $x'\in D$.
\end{conj}



\subsection{Remarks on free boundary regularity}\label{sec:fbpdisc}
Let $u$ and $v$ be like in Theorem \ref{thm:fbp}. From Lemma \ref{lem:n-1dim} and (\ref{main2}) it follows that the function $v$ is the solution of the following obstacle problem
\begin{equation}\label{eqblanc}
\Delta v=\chi_{\{v>0\}} h(x'),
\end{equation}
where
$$
h(x')=1-\partial_\nu u (x',0) - \partial_\nu u (x',1)\in C^\alpha(D).
$$

In the points of the free boundary $ x'\in\partial \{v>0\} \cap D$, where $h(x')>0$ we can apply the Theorem 7.2 in \cite{Bl} and obtain that

\vspace{2mm}

\noindent either
\vspace{1mm}

$\bullet$ $x'$ is a regular point and the free boundary is $C^{1,\alpha}$ smooth,

\vspace{1mm}

\noindent or
\vspace{1mm}

$\bullet$ $x'$ is a singular point, i.e. $\lim_{r\to 0}\frac{|\{v=0\}\cap B_r(x')|}{|B_r(x')|}=0$, and the free

boundary in the ball $B_r(x')$ has a minimum diameter less than $\sigma(r)$,

for some given
modulus of continuity $\sigma$.

\vspace{4mm}

Observe that in general it is possible to have singular
points of the free boundary where $h(x')=0$ and the free boundary is not flat. One can take for example
the function $u(x_1,x_2)=\chi_{\{x_1x_2>0\}}x_1^2x_2^2$ and obtain a cross-shaped free boundary, with its minimal diameter scaling of order $r$ in $B_r(0)$.

Whether this kind of non-flat singularities can be excluded for the minimizers of (\ref{funcobst})
is a subject of ongoing research.



\bibliographystyle{plain}
\bibliography{rearr}


\end{document}